\documentclass[twoside,12pt]{cmft}

\pagespan{1}{?}
\cmftinfo{XXYYYZZ}

\newtheorem{theorem}{Theorem}[section]

\newtheorem{corollary}[theorem]{Corollary}
\newtheorem{lemma}[theorem]{Lemma}

\newtheorem{proposition}[theorem]{Proposition}
\newtheorem{example}[theorem]{Example}

\numberwithin{equation}{section}

\newcommand{\R}{\mathbb{R}}
\newcommand{\N}{\mathbb{N}}

\newcommand{\C}{\mathbb{C}}
\newcommand{\D}{\mathbb{D}}

\newcommand{\supp}{\textup{supp}}

\begin{document}

\title[How to find a measure from its potential]
{How to find a measure from its potential}
\date{December 8, 2007}

\author[Pritsker]{Igor E. Pritsker}
\email{igor@math.okstate.edu}
\address{Department of Mathematics, Oklahoma State University,
Stillwater, OK 74078, U.S.A.}



\thanks{Research was partially supported by the National Security
Agency under grant H98230-06-1-0055, and by the Alexander von
Humboldt Foundation.}

\dedicatory{Dedicated to Professor W. K. Hayman on his 80th
birthday}

\begin{abstract}

We consider the problem of finding a measure from the given values
of its logarithmic potential on the support. It is well known that a
solution to this problem is given by the generalized Laplacian. The
case of our main interest is when the support is contained in a
rectifiable curve, and the measure is absolutely continuous with
respect to the arclength on this curve. Then the generalized
Laplacian is expressed by a sum of normal derivatives of the
potential. Such representation was available for smooth curves, and
we show it holds for any rectifiable curve in the plane. We also
relax the assumptions imposed on the potential.

Finding a measure from its potential often leads to another closely
related problem of solving a singular integral equation with Cauchy
kernel. The theory of such equations is well developed for smooth
curves. We generalize this theory to the class of Ahlfors regular
curves and arcs, and characterize the bounded solutions on arcs.

\end{abstract}

\maketitle

\section{Logarithmic potentials in the complex plane}

Let $\mu$ be a positive Borel measure with the compact support
$\supp\,\mu$ in the complex plane $\C.$ Consider its logarithmic
potential
\begin{align} \label{1.1}
u(z):= \int \log|z-t|\,d\mu(t),
\end{align}
which is a subharmonic function in $\C,$ see, e.g., Hayman and
Kennedy \cite{HK} and Ransford \cite{Ran}. Furthermore, $u$ is
harmonic in $\C\setminus\supp\,\mu.$ Suppose that we know the values
of the potential function $u$ on a certain set. How can one recover
the measure from those values? A general answer to this question is
well known \cite{HK,Ran} in terms of the generalized
(distributional) Laplacian, and is a part of the Riesz
Representation Theorem for subharmonic functions (cf. Chapter 3 of
\cite{HK} and Section 3.7 of \cite{Ran}).

\noindent{\bf Theorem A.} {\em Suppose that $\mu$ is a positive
Borel measure with compact support $\supp\,\mu\subset\C$, whose
potential $u$ is defined by \eqref{1.1}. Then the measure is
represented by
\begin{align} \label{1.2}
d\mu = \frac{1}{2\pi}\Delta\,u,
\end{align}
where $\Delta$ is the generalized Laplacian.}

While this theorem gives a very general answer, it is not always
easy to apply in specific problems. The generalized Laplacian may,
in particular, take quite different forms. The most known and
natural case is when the generalized Laplacian reduces to the
regular one, under the additional $C^2$ smoothness assumptions on
the potential $u$ (see Theorem 1.3 in Saff and Totik \cite[p.
85]{ST}).

\noindent{\bf Theorem B.} {\em If $u$ has continuous second partial
derivatives in a domain $D\subset\C,$ then $\mu$ is absolutely
continuous with respect to the area Lebesgue measure $dxdy$ in $D$,
and
\begin{align} \label{1.3}
d\mu(x,y) = \frac{1}{2\pi}\Delta\,u(z)\,dxdy,\quad z=x+iy\in D,
\end{align}
where $\Delta$ is the regular Laplacian.}

We remark that the $C^2$ assumption on $u$ may be relaxed if one
follows a standard proof as given in \cite{ST}, but uses a more
general version of Green's theorem found in Shapiro \cite{Sha} or
Cohen \cite{Coh} (see also Bochner \cite{Boc}). In particular, it is
sufficient to assume in Theorem B that the second partial
derivatives of $u$ exist and are integrable in $D$ with respect to
$dxdy.$

Another well known example of the generalized Laplacian is given by
a sum of points masses. If $p(z)=\prod_{j=1}^n (z-a_j)$ is any
polynomial and $u(z)=\log |p(z)|$, then $\mu = \sum_{j=1}^n
\delta_{a_j}$, where $\delta_{a_j}$ is a unit point mass at
$a_j\in\C$ (cf. Theorem 3.7.8 in \cite[p. 76]{Ran}).

We are mostly interested in the case when $\mu$ is supported on a
rectifiable curve (or arc), and is absolutely continuous with
respect to the arclength measure on this curve. Then the form of the
result expressed through the normal derivatives of potential is also
classical, while it is certainly difficult to find the original
source. We follow the statement of Theorem  1.5 in \cite[p. 92]{ST}.

\noindent{\bf Theorem C.} {\em Suppose that the intersection of
$\supp\,\mu$ with a domain $D\subset\C$ is a simple $C^{1+\delta}$
arc, $\delta>0,$ with the left normal $n_+$ and the right normal
$n_-.$ If the potential $u$ is \textup{Lip\,1} in a neighborhood of
this arc then $\mu$ is absolutely continuous with respect to the
arclength $ds$ on this arc and
\begin{align} \label{1.4}
d\mu(s) = \frac{1}{2\pi}\left( \frac{\partial u}{\partial n_+}(s) +
\frac{\partial u}{\partial n_-}(s) \right)ds.
\end{align}}

Attempts to remove smoothness assumptions imposed on the curve go
back to Plemelj \cite{Ple1} and Radon \cite{Rad}, where the idea of
flux of a potential along a curve was introduced. This approach was
further developed by Burago, Maz'ya and Sapozhnikova \cite{BMS}, and
by Burago and Maz'ya \cite{BM}. We show that \eqref{1.4} remains
valid for an arbitrary rectifiable curve, and also relax assumptions
on the potential $u$. In order to state our assumptions, we need to
introduce the Smirnov spaces of analytic and harmonic functions.

Let $\Gamma\subset\C$ be a closed Jordan rectifiable curve. The
complement of $\Gamma$ in $\overline\C$ is the union of a bounded
Jordan domain $D_+$ and an unbounded Jordan domain $D_-.$ Consider a
conformal mapping $\psi:\D \to D_+$, where $\D:=\{w:|w|<1\},$ and
define the level curves $\Gamma_r:=\{z=\psi(w)\in D_+: |w|=r\},\
0\le r<1.$ A function $f$ analytic in $D_+$ is said to belong to the
Smirnov space $E^p(D_+),\ 1\le p<\infty,$ if
\begin{align} \label{1.5}
\sup_{0\le r<1} \int_{\Gamma_r} |f(z)|^p\,|dz| < \infty,
\end{align}
see Chapter 10 of Duren \cite{Dur}. These spaces are natural analogs
of the Hardy spaces $H^p$ on the unit disk $\D.$ Furthermore, $f\in
E^p(D_+)$ if and only if $f(\psi(w))(\psi'(w))^{1/p}\in H^p,$ cf.
\cite[p. 169]{Dur}. A function from $E^p(D_+)$ has nontangential
limit values almost everywhere on $\Gamma$ with respect to the
arclength measure by Theorem 10.3 of \cite[p. 170]{Dur}, and these
values are in $L^p(\Gamma,ds).$ We similarly introduce the harmonic
Smirnov space $e^p(D_+)$ that consists of harmonic functions in
$D_+$ satisfying \eqref{1.5}. Using a conformal mapping $\Psi:\D\to
D_-$, we repeat the same steps to define the corresponding spaces
$E^p(D_-)$ and $e^p(D_-)$ for the unbounded domain $D_-.$

Assume that $\supp\,\mu\subset\Gamma.$ Recall that the potential $u$
is harmonic in $\C\setminus\supp\, \mu$. Hence its partial
derivatives $u_x$ and $u_y$ are harmonic in $\C\setminus\supp\, \mu$
too. We say that $\nabla u \in e^p(D_{\pm})$ if both $u_x,u_y\in
e^p(D_+)$ and $u_x,u_y\in e^p(D_-)$ hold. Let $n_+$ be the inner
normal vector (pointing into $D_+$), and $n_-$ be the outer normal
vector (pointing into $D_-$) on $\Gamma.$ The normal direction is
well defined on $\Gamma$ for almost every point with respect to the
arclength measure. Hence we can define the directional derivatives
$\partial u/\partial n_+$ in $D_+$ and $\partial u/\partial n_-$ in
$D_-$. The corresponding boundary limit values for these normal
derivatives exist a.e. on $\Gamma$, as we show in the proof of the
theorem stated below. Thus \eqref{1.4} is understood in the sense of
such boundary values.

\begin{theorem} \label{thm1.1}
Let $\Gamma\subset\C$ be an arbitrary rectifiable Jordan curve.
Suppose that $\supp\,\mu\subset\Gamma$ and $u$ is the potential of
$\mu$ defined by \eqref{1.1}. If $\nabla u \in e^1(D_{\pm})$ then
$\mu$ is absolutely continuous with respect to the arclength $ds$ on
$\Gamma$ and \eqref{1.4} holds.
\end{theorem}

We present a ``complex function theory" proof of Theorem
\ref{thm1.1} in Section 3, which is based on the Cauchy transform of
$\mu.$ It is possible to extend Theorem C by following the
conventional proof of \cite{ST}, and by employing a version of
Green's theorem from \cite{Sha}. However, this gives a less general
result than the one in Theorem \ref{thm1.1}. Note that if $u$ is
Lipschitz continuous in an open neighborhood $G$ of $\supp\,\mu$,
then the partial derivatives of $u$ are bounded on
$G\setminus\Gamma$ by the Lipschitz constant. But they are also
bounded in a neighborhood of $\Gamma\setminus G$, being harmonic
outside $\supp\,\mu$. Hence the assumptions of Theorem C imply that
$u_x$ and $u_y$ are bounded in both domains $D_+$ and $D_-$ by the
maximum principle. It immediately follows that $\nabla u \in
e^1(D_{\pm})$, i.e. our assumption on $u$ is indeed weaker.

We also remark that Theorem \ref{thm1.1} has local nature, in fact.
If the intersection of $\supp\,\mu$ with a domain $G\subset\C$ is a
Jordan arc $\gamma$, then we write $\mu_1:=\mu\vert_{\gamma}$ and
$\mu_2:=\mu\vert_{\supp\,\mu\setminus\gamma}$. Hence
\[
u(z) = \int \log|z-t|\,d\mu_1(t) + \int \log|z-t|\,d\mu_2(t),
\]
and Theorem \ref{thm1.1} is applicable to the first potential in the
above equation for recovering the measure on $\gamma$, provided $u$
satisfies the assumptions. Note that the second potential is
harmonic in $G$, so that it gives no contribution to the measure on
$\gamma$ by Theorem A (or B, or C).

One of the most natural applications for Theorem \ref{thm1.1} is to
the equilibrium potential of a compact set $E\subset\C.$ If $E$ is
not polar, then the equilibrium measure $\mu_E$ exists and is a
unique positive Borel measure of mass one, whose potential $u_E$ is
equal to a constant $V_E$ everywhere on $E$, with a possible
exception of a polar subset (cf. \cite{HK,Ran}). Furthermore, if
$E:=D_+\cup\Gamma$ is the closure of a Jordan domain, then
$u_E(z)=V_E$ for all $z\in E,$ and $\supp\,\mu_E = \Gamma$, see
\cite{HK,Ran} and Theorem B. Let $g_E$ be the Green function of
$D_-,$ with pole at infinity. Then $u_E(z)=V_E+g_E(z)$ and
$g_E(z)=\log|\Phi(z)|,\ z\in D_-$, where $\Phi:D_-\to \{w:|w|>1\}$
is a conformal map satisfying $\Phi(\infty)=\infty$. Since $\Phi\in
E^1(D_-)$ for a rectifiable $\Gamma,$ see \cite{Dur,Gol,Pom}, we
have that $\nabla u_E = \nabla g_E \in e^1(D_-)$. Obviously, $\nabla
u_E =(0,0) \in e^1(D_+)$, because $u_E$ is constant in this domain.
Hence all assumptions of Theorem \ref{thm1.1} are satisfied. Also,
$\partial u_E/\partial n_+$ has zero boundary values a.e. on
$\Gamma$ and \eqref{1.4} takes the following familiar form.

\begin{example} \label{ex1.2}
Let $E\subset\C$ be the closure of a Jordan domain $D_+$ bounded by
a rectifiable Jordan curve $\Gamma.$ The equilibrium measure $\mu_E$
for $E$ is absolutely continuous with respect to the arclength
measure $ds$ on $\Gamma,$ and
\[
d\mu_E = \frac{1}{2\pi} \frac{\partial g_E}{\partial n_-} \,ds.
\]
\end{example}

More examples of applications to the equilibrium measures for energy
problems with external fields may be found in
\cite{ST,DKM,Pr05a,Pr05b} (see also references therein).

The problem of finding a measure from its potential is of interest
in higher dimensions too. For example, consider a positive Borel
measure $\sigma$ compactly supported in $\R^3$, and define its
Newtonian potential by
\[
U(x):=\int \frac{d\,\sigma(y)}{|x-y|},\quad x\in\R^3,
\]
where $|x-y|$ is the Euclidean distance between $x,y\in\R^3.$
Clearly, $U$ is superharmonic in $\R^3$ and harmonic in
$\R^3\setminus\supp\,\sigma,$ see \cite{HK} and \cite{Kel}. If $U\in
C^2(D)$ for a domain $D$ in $\R^3,$ then an analog of Theorem B
gives \cite[p. 156]{Kel}
\[
d\sigma = - \frac{1}{4\pi} \Delta U\,dV,
\]
where $dV$ is the volume measure. When $\sigma$ is supported on a
sufficiently smooth surface $S$, we have an analog of Theorem C
\cite[p. 164]{Kel}
\[
d\sigma = - \frac{1}{4\pi} \left( \frac{\partial U}{\partial n_+} +
\frac{\partial U}{\partial n_-} \right)dS,
\]
where $dS$ is the surface area measure, and $n_+,\ n_-$ are the
inner and the outer normals to $S$. Clearly, one should be able to
relax smoothness assumptions for the surface $S$, but a natural
analog of rectifiable curve (cf. Theorem \ref{1.1}) in this setting
is not obvious at all.

\section{Singular integral equations with Cauchy kernel}

Define the Cauchy transform of the measure $\mu$ by
\begin{align*}
C\mu(z) := \frac{1}{2\pi i} \int \frac{d\mu(t)}{t-z},\quad
z\in\C\setminus\supp\,\mu,
\end{align*}
which is an analytic in $\overline\C\setminus\supp\,\mu$ function
such that $C\mu(\infty)=0.$ It is well known that $C\mu$ is closely
related to the potential $u$ of \eqref{1.1}. Indeed, if we consider
a multivalued analytic function $F(z):= \int \log(z-t) \,d\mu(t)$,
then $u=\Re(F)$ and $C\mu = -F'$ in $\C\setminus\supp\,\mu.$ In
fact, these ideas are used in the proof of Theorem \ref{thm1.1}, see
the argument beginning with \eqref{3.1}. More discussion and history
of such relations may be found in Muskhelishvili \cite{Mus} and
Danilyuk \cite{Dan}, see also the work of Plemelj \cite{Ple1,Ple2},
Radon \cite{Rad} and Bertrand \cite{Ber}.

Suppose that $\supp\,\mu\subset\Gamma,$ where $\Gamma$ is a Jordan
rectifiable curve of length $l.$ We keep the same notation $D_+$ for
the bounded component of the complement of $\Gamma$, and $D_-$ for
the unbounded one. Let $\Gamma$ be parametrized by $z=z(s)$, where
$s\in[0,l]$ is the arclength parameter and $z'(s)$ is the unit
tangent vector to the curve. Recall that the tangent and normal
vectors exist almost everywhere on $\Gamma$ with respect to the
arclength measure. Assume further that $\mu$ is absolutely
continuous with respect to $ds$ with the density $f(z(s)) z'(s)$,
where $f\in L^1(\Gamma,ds).$ Then we can consider the singular
Cauchy integral of $f$ defined by
\begin{align} \label{2.1}
Sf(z) := \frac{1}{\pi i} \int_{\Gamma} \frac{f(t)dt}{t-z} =
\lim_{\varepsilon\to 0} \frac{1}{\pi i}
\int_{\Gamma_{\varepsilon}(z)} \frac{f(t)dt}{t-z},\quad z \in\Gamma,
\end{align}
where $\Gamma_{\varepsilon}(z):=\{t\in\Gamma:|t-z|\ge\varepsilon\}$,
i.e. the integral is understood as the Cauchy principal value, cf.
\cite{Mus}, \cite{Gol} and \cite{Dan}. Existence of $Sf$ is subject
to appropriate conditions on the function $f$ and the curve
$\Gamma$. For the Cauchy transform
\begin{align} \label{2.2}
Cf(z) = \frac{1}{2\pi i} \int_{\Gamma} \frac{f(t)\,dt}{t-z},\quad
z\in\C\setminus\Gamma,
\end{align}
let $Cf_+(\zeta)$ (respectively $Cf_-(\zeta)$) denote the
nontangential limit value from $D_+$ (respectively from $D_-$) at a
point $\zeta\in\Gamma$. A classical and important result of Privalov
\cite{Pri} gives connections between $Sf$ and the boundary values of
$Cf$.

{\bf Privalov's Fundamental Lemma.} {\em The nontangential boundary
limit values $Cf_+$ (or $Cf_-$) exist a.e. on $\Gamma$ if and only
if $Sf$ exists a.e. on $\Gamma.$ Furthermore, in the case of a.e.
existence, the Plemelj-Sokhotski formulas
\begin{align} \label{2.3}
Cf_+(z) - Cf_-(z) = f(z) \quad \mbox{and} \quad Cf_+(z) + Cf_-(z) =
Sf(z)
\end{align}
hold for a.e. $z\in\Gamma.$}

For the proof and thorough discussion of this result, we refer to
\cite{Pri,Gol,Dan}. If $1<p<\infty$, then another fundamental result
of David \cite{Dav} states that $S:L^p(\Gamma,ds)\to L^p(\Gamma,ds)$
is a bounded operator if and only if $\Gamma$ is Ahlfors regular.
The Ahlfors regularity condition means that there is a constant
$A>0$ such that for any disk $D_r$ of radius $r$ we have
\[
|\Gamma\cap D_r| \le A r,
\]
where $|\Gamma\cap D_r|$ is the length of the intersection. This
class of curves is sufficiently wide as it allows any angles (even
cusps), see Chapter 7 of Pommerenke \cite{Pom} for more on geometry.
A Jordan arc is said to be Ahlfors regular if it is a subarc of an
Ahlfors regular curve. In the sequel, we shall always make a natural
assumption that $\Gamma$ is Ahlfors regular, to insure the a.e.
existence of $Sf\in L^p(\Gamma,ds)$ for $f\in L^p(\Gamma,ds)$, and
the validity of \eqref{2.3}. If $f$ belongs to the class
$H_{\alpha}(\Gamma),\ 0<\alpha<1,$ of H\"older continuous functions
on $\Gamma,$ then $Sf\in H_{\alpha}(\Gamma)$ for Ahlfors regular
$\Gamma.$ This generalization of the Plemelj-Privalov theorem was
proved by Salaev \cite{Sal}. A complete description of curves that
allow the Cauchy singular operator to preserve moduli of continuity
is contained in Guseinov \cite{Gus}.

Singular integral equations with Cauchy kernel arise naturally in
the problem of finding a measure from its potential. For example,
the equation
\begin{align} \label{2.4}
Sf(z) = \frac{1}{\pi i} \int_{\Gamma} \frac{f(t)\,dt}{t-z} = g(z)
\end{align}
was used repeatedly to find the weighted equilibrium measures in
\cite{ST,DKM,MS,Pr05a,Pr05b} and many other papers. Here, the
function $g$ is either obtained by differentiation of the known
values for the potential $u$ on the support of $\mu$, or found from
the Plemelj-Sokhotski formula \eqref{2.3} as $g(z)=Cf_+(z) +
Cf_-(z).$ The solution of \eqref{2.4} for a closed contour $\Gamma$
is well known (see, e.g., \cite[\S 27]{Mus}), and represents the
self-inversive property of the operator $S.$

\begin{proposition} \label{prop2.1}
Suppose that $\Gamma$ is an Ahlfors regular closed Jordan curve, and
that $g\in L^p(\Gamma,ds),\ 1<p<\infty$. The equation $Sf=g$ on
$\Gamma$ has the unique solution $f=Sg \in L^p(\Gamma,ds)$.
\end{proposition}

In many applications to recovering a measure from its potential,
\eqref{2.4} holds on the support of the measure, which may be
different from a closed curve. A rather common situation is when the
support consists of several arcs, see \cite{ST,DKM,MS,Pr05a,Pr05b}.
Let $L:=\cup_{j=1}^N \gamma(a_j,b_j)$ be the union of $N$ disjoint
Ahlfors regular arcs $\gamma(a_j,b_j)$ with endpoints $a_j$ and
$b_j.$ For a function $f\in L^1(L,ds),$ we define the Cauchy
singular integral operator $S_L f$ on $L$ similarly to \eqref{2.1}.
We can assume that $L\subset\Gamma,$ where $\Gamma$ is an Ahlfors
regular closed Jordan curve with interior $D_+$ and exterior $D_-.$
It is always possible to extend $f$ from $L$ to $\Gamma$ by letting
$f(z) = 0,\ z\in\Gamma\setminus L,$ and view the operator $S_L$ as a
restriction of $S$, by setting $S_Lf=Sf$ for $f\in L^1(\Gamma,ds), \
f\vert_{\Gamma\setminus L}\equiv 0.$ Hence Privalov's Fundamental
Lemma and many other facts easily carry over to the case of $S_L$.

For $R(z):=\prod_{j=1}^N (z-a_j)(z-b_j),$ we consider the branch of
$\sqrt{R(z)}$ defined in the domain $\C\setminus L$ by
$\lim_{z\to\infty} \sqrt{R(z)}/z^N = 1.$ By the values of
$\sqrt{R(z)}$ on $L$, we understand the boundary limit values from
$D_+.$ A general solution of the equation $S_Lf=g$ for H\"older
continuous $g$ on smooth arcs was first found by Muskhelishvili
\cite[Chap. 11]{Mus}. The case of $L^p$ solutions was later
considered by Hvedelidze \cite{Hve}. We generalize their ideas to
prove the following.

\begin{theorem} \label{thm2.2}
Let $L:=\cup_{j=1}^N \gamma(a_j,b_j)$ be a union of disjoint Ahlfors
regular arcs, and let $R(z):=\prod_{j=1}^N (z-a_j)(z-b_j).$ If $g\in
L^p(L,ds),\ 2<p<\infty$, then any solution of the equation $S_Lf=g$
in $L^1(L,ds)$ has the form
\begin{align} \label{2.5}
f(z) = \frac{1}{\pi i \sqrt{R(z)}} \int_{L}
\frac{g(t)\sqrt{R(t)}\,dt}{t-z} + \frac{P_{N-1}(z)}{\sqrt{R(z)}}
\quad \mbox{a.e. on }L,
\end{align}
where $P_{N-1}\in\C_{N-1}[z].$
\end{theorem}
Here, $\C_{N-1}[z]$ denotes the set of polynomials with complex
coefficients of degree at most $N-1.$

It is of interest that certain solutions may also be written in a
different form.

\begin{corollary} \label{cor2.3}
Let $L$ and $g$ be as in Theorem \ref{thm2.2}. The function
\begin{align} \label{2.6}
f_0(z) = \frac{\sqrt{R(z)}}{\pi i} \int_{L}
\frac{g(t)\,dt}{\sqrt{R(t)}(t-z)}, \quad z\in L,
\end{align}
is a solution of $S_Lf=g$ in $L^1(L,ds)$ if and only if
\begin{align} \label{2.7}
\int_{L} \frac{t^k g(t)\,dt}{\sqrt{R(t)}} = 0,  \qquad
k=0,\ldots,N-1.
\end{align}
\end{corollary}

If the right hand side of the equation $S_Lf=g$ is H\"older (or
Dini) continuous, then the continuity properties are also preserved
for the solutions, provided that we stay away from the endpoints of
$L.$ This follows from the corresponding results for the operator
$S$ on closed contours, see \cite{Sal,Gus}.

\begin{corollary} \label{cor2.4}
Let $L$ be as in Theorem \ref{thm2.2}. If $g \in H_{\alpha}(L),\
0<\alpha<1,$ then for any compact set $E\subset L\setminus
\{a_j,b_j\}_{j=1}^N$ the general solution \eqref{2.5} belongs to
$H_{\alpha}\left(E\right).$
\end{corollary}

It is important for applications to find the bounded solution of the
equation $S_Lf=g$, and describe conditions for its existence. For
example, bounded solutions play a special role in finding the
weighted equilibrium measures for ``good" weights
\cite{ST,DKM,MS,Pr05a,Pr05b}.

\begin{corollary} \label{cor2.5}
Assume that $L$ satisfies the conditions of Theorem \ref{thm2.2},
and that $g \in H_{\alpha}(L),\ \alpha>0.$ A bounded solution
$f_0\in L^{\infty}(L,ds)$ for $S_Lf=g$ exists if and only if
\eqref{2.7} holds true.

Furthermore, if \eqref{2.7} is satisfied, then
\begin{align} \label{2.8}
f_0(z) = \frac{\sqrt{R(z)}}{\pi i} \int_{L}
\frac{g(t)\,dt}{\sqrt{R(t)}(t-z)}, \quad z\in L,
\end{align}
where $f_0\in C(L)$ and $f_0(a_j)=f_0(b_j)=0,\ j=1,\ldots,N.$
\end{corollary}

Several applications of Corollary \ref{cor2.5} on arcs of the unit
circle may be found in \cite{Pr05b}, see Theorems 1.5, 2.1 and 2.2.
In particular, those results describe the explicit forms of
equilibrium measures with external fields defined by the exponential
and polynomial weights.

\section{Proofs}

\begin{proof}[Proof of Theorem \ref{thm1.1}]

Let $v^+$ be a harmonic conjugate of $u$ in $D_+,$ so that
$F_+:=u+iv^+$ is analytic in $D_+.$ By the Cauchy-Riemann equations,
we have $u_x(z)=v^+_y(z)$ and $u_y(z)=-v^+_x(z)$ for all $z\in D_+.$
Hence $v^+_x,v^+_y \in e^1(D_+)$ and $F_+'=u_x+iv^+_x\in E^1(D_+).$
It immediately follows that the nontangential limit values of $F_+'$
exist almost everywhere on $\Gamma$ with respect to the arclength
measure, and the same is true for $u_x,u_y,v^+_x,v^+_y.$ Since
$\Gamma$ is rectifiable, the tangent and normal vectors exist at
almost every $z\in\Gamma$. Therefore, for almost every $z\in\Gamma$,
we {\em simultaneously} have the normal vector and the nontangential
limit values for $u_x$ and $u_y$. If $n_+=(-\sin\theta,\cos\theta)$
is the inner unit normal (pointing inside $D_+$) at such a point
$z$, then the derivative in the direction $n_+$ at any point
$\zeta\in D_+$ is $\partial u/\partial n_+(\zeta) = - u_x \sin\theta
+ u_y \cos\theta.$ Thus we can define the limit boundary values of
$\partial u/\partial n_+(z)$ for a.e. $z\in\Gamma$ and
\[
\int_{\Gamma} \left|\frac{\partial u}{\partial n_+}(z(s))\right| ds
\le \int_{\Gamma} \left|u_x(z(s))\right| ds + \int_{\Gamma}
\left|u_y(z(s))\right| ds < \infty.
\]
Applying the same argument to a conjugate $v^-$ and the analytic
completion $F_-:=u+iv_-$ for $u$ in $D_-,$ we obtain that $F_-' \in
E^1(D_-)$ and $\partial u/\partial n_- \in L^1(\Gamma,ds).$ Note
that $F_-$ is multi-valued, in general, but $F_-'$ is single-valued.

The Cauchy-Riemann equations imply, after passing to the boundary
values, that
\[
\frac{\partial u}{\partial n_+}(z)=-\frac{\partial v^+}{\partial
s}(z) \quad\mbox{and}\quad \frac{\partial u}{\partial
n_-}(z)=\frac{\partial v^-}{\partial s}(z) \qquad \mbox{for a.e. }
z\in\Gamma.
\]
Therefore, we have a.e. on $\Gamma$ that
\begin{align} \label{3.1}
\frac{\partial u}{\partial n_+}(z) + \frac{\partial u}{\partial
n_-}(z) &= \frac{\partial v^-}{\partial s}(z) - \frac{\partial
v^+}{\partial s}(z) = \Im\left( \frac{\partial F_-}{\partial s}(z) -
\frac{\partial F_+}{\partial s}(z) \right) \\ \nonumber &= \Im\left(
[F_-'(z) - F_+'(z)]z'(s) \right),
\end{align}
where $z'(s)$ is the unit tangent vector to $\Gamma$ at $z(s).$

Introducing the multivalued function $\int \log(z-t)\,d\mu(t)$ with
the real part $u(z)$, we observe that $F_+$ and $F_-$ are branches
of this function and
\[
F_{\pm}'(z) = \int \frac{d\mu(t)}{z-t},\quad z\in D_{\pm}.
\]
Consider the Cauchy transform of $\mu$
\[
C\mu(z) = \frac{1}{2\pi i} \int \frac{d\mu(t)}{t-z},\quad z\in
\overline{\C}\setminus \supp\,\mu.
\]
Since $2\pi i\,C\mu(z)=-F_{\pm}'(z),\ z\in D_{\pm},$ the
nontangential limit values $C\mu_+$ from $D_+$ and $C\mu_-$ from
$D_-$ exist a.e. on $\Gamma.$ The Fundamental Lemma of Privalov on
the Cauchy singular integral for measures (cf. Privalov \cite[pp.
183-189]{Pri} and Danilyuk \cite[pp. 118-125]{Dan}) gives that the
Plemelj-Sokhotski formula
\[
[C\mu_+(z(s)) - C\mu_-(z(s))] z'(s) = \frac{d \mu}{d s}(s)
\]
holds a.e. on $\Gamma,$ where $d \mu/d s$ is the density of the
absolutely continuous part of $\mu.$ Note that the right hand side
is real. It now follows from \eqref{3.1} that
\[
\frac{1}{2\pi} \left( \frac{\partial u}{\partial n_+}(z(s)) +
\frac{\partial u}{\partial n_-}(z(s)) \right) = [C\mu_+(z(s)) -
C\mu_-(z(s))] z'(s) \quad \mbox{a.e. on }\Gamma.
\]
Let $\nu$ be the measure supported on $\Gamma,$ which is absolutely
continuous with respect to $ds$, and whose density is defined by the
left hand side of the above equation (cf. \eqref{1.4}). We shall
show that $\mu=\nu.$ Recall that $F_+'\in E^1(D_+)$ and $F_-'\in
E^1(D_-)$ with $F_-'(\infty)=0.$ Using Cauchy's integral formula, we
obtain by Theorem 10.4 of \cite[p. 170]{Dur} that
\begin{align*}
\int \frac{d\nu(t)}{t-z} &= \int_{\Gamma} \frac{(C\mu_+(t) -
C\mu_-(t)) dt}{t-z} = \frac{1}{2\pi i} \int_{\Gamma} \frac{(F_-'(t)
- F_+'(t)) dt}{t-z} \\  &= - F_{\pm}'(z) = 2\pi i\, C\mu(z),\quad
z\in D_{\pm}.
\end{align*}
Hence the Cauchy transforms of $\nu$ and $\mu$ coincide, i.e.,
\[
\int \frac{d(\nu-\mu)(t)}{t-z} = 0,\quad z\in D_{\pm}.
\]
Expanding $1/(t-z)$ in a series of negative powers of $z$ around
$z=\infty,$ we obtain in a standard way that
\[
\int t^n\,d(\nu-\mu)(t) = 0,\quad n=0,1,2,\ldots.
\]
Similarly, expanding the kernel $1/(t-z)$ near a fixed point $z_0\in
D_+,$ we have
\[
\int \frac{d(\nu-\mu)(t)}{(t-z_0)^n} = 0,\quad n\in\N.
\]
But the span of the function system $\{t^n\}_{n=0}^{\infty} \bigcup
\{(t-z_0)^{-n}\}_{n=1}^{\infty}$ is dense in $C(\Gamma),$ see
Chapter 3, \S 1 of Gaier \cite{Gai}. Hence
\[
\int f(t)\,d(\nu-\mu)(t) = 0,
\]
for any continuous function $f$ on $\Gamma$. Since $\nu-\mu$ is
orthogonal to all continuous functions on its support, it must
vanish identically.

\end{proof}

\begin{proof}[Proof of Proposition \ref{prop2.1}]

Observe that for $g\in L^p(\Gamma,ds)$ we also have that $Sg \in
L^p(\Gamma,ds)$ by \cite{Dav}.  Hence we have from \eqref{2.3} that
$Cg_+ - Cg_- = g$ and $Cg_+ + Cg_- = Sg$ a.e. on $\Gamma.$ Note that
the function $H_+(z):=Cg(z), \ z\in D_+,$ is analytic in $D_+$ and
has the boundary values $Cg_+ \in L^p(\Gamma,ds)$. It follows that
$H_+\in E^p(D_+)$ and we obtain by Theorem 10.4 of \cite{Dur} that
\begin{align} \label{3.2}
C(Cg_+)(z) = \frac{1}{2\pi i} \int_{\Gamma} \frac{Cg_+(t)\,dt}{t-z}
= \left\{
\begin{array}{ll}
H_+(z), \quad &z\in D_+, \\
0, \quad &z\in D_-.
\end{array}
\right.
\end{align}
Thus we also have for $H_-(z):=Cg(z), \ z\in D_-,$ that $H_-\in
E^p(D_-)$ and $H_-(\infty)=0.$ Hence
\begin{align} \label{3.3}
C(Cg_-)(z) = \frac{1}{2\pi i} \int_{\Gamma} \frac{Cg_-(t)\,dt}{t-z}
= \left\{
\begin{array}{ll}
-H_-(z), \quad &z\in D_-, \\
0, \quad &z\in D_+,
\end{array}
\right.
\end{align}
where the integral is taken in the positive direction with respect
to $D_+.$

Applying \eqref{2.3} to $h:=Sg,$ we have that $Ch_+ - Ch_- = h$ and
$Ch_+ + Ch_- = Sh$ a.e. on $\Gamma,$ where $Sh \in L^p(\Gamma,ds).$
Consider the Cauchy transform $Ch(z)$ for $z\in D_+$ and use
\eqref{3.2}-\eqref{3.3} to evaluate
\[
Ch(z) = \frac{1}{2\pi i} \int_{\Gamma} \frac{Sg(t)\,dt}{t-z} =
\frac{1}{2\pi i} \int_{\Gamma} \frac{(Cg_+(t)+Cg_-(t))\,dt} {t-z} =
H_+(z).
\]
Similarly, we obtain for $z\in D_-:$
\[
Ch(z) = \frac{1}{2\pi i} \int_{\Gamma} \frac{Sg(t)\,dt}{t-z} =
\frac{1}{2\pi i} \int_{\Gamma} \frac{(Cg_+(t)+Cg_-(t))\,dt} {t-z} =
-H_-(z).
\]
Hence $Sh = Ch_+ + Ch_- = Cg_+ - Cg_- = g$ a.e. on $\Gamma,$ so that
$h=Sg$ is a solution. This also implies that $S^2g=S(Sg)=g$ for any
$g \in L^p(\Gamma,ds).$ If we assume that there are two solutions
$f_1$ and $f_2$, then applying $S$ to $Sf_1=Sf_2=g$ gives
$f_1=f_2=Sg.$

\end{proof}

For the proof of Theorem \ref{thm2.2}, we need a lemma that
describes all solutions of the homogeneous equation $S_L f = 0$ on
$L$. It follows from the lemma below that the kernel of the operator
$S_L$ has dimension $N$, while the kernel of $S$ is trivial in the
case of a closed curve, by the previous proof.

\begin{lemma} \label{lem3.1}
Let $L:=\cup_{j=1}^N \gamma(a_j,b_j)$ be a union of disjoint Ahlfors
regular arcs, and let $R(z):=\prod_{j=1}^N (z-a_j)(z-b_j).$ All
solution of the equation $S_L f = 0$ in $L^1(L,ds)$ are given by the
functions $f=P_{N-1}/\sqrt{R},$ where $P_{N-1}\in\C_{N-1}[z]$.
\end{lemma}

\begin{proof}

For any $P_{N-1}\in\C_{N-1}[z]$, we first show that
\begin{equation} \label{3.4}
\frac{1}{\pi i} \int_L \frac{P_{N-1}(t)\,dt}{(t-z) \sqrt{R(t)}} =
\left\{
\begin{array}{ll}
0 \quad &\mbox{for a.e. }z\in L, \\
P_{N-1}(z)/\sqrt{R(z)}, \ &z \in \C\setminus L,
\end{array}
\right.
\end{equation}
where the integral is understood in the Cauchy principal value sense
for $z\in L.$ Let $h(z) = P_{N-1}(z)/\sqrt{R(z)}, \ z\in \C\setminus
L$. It is clear that the limit values of $\sqrt{R(z)}$ as $z$ tends
to $\zeta\in L$ from $D_+$ and from $D_-$ are negatives of each
other. Hence the boundary limit values of $h$ on $L$ satisfy
\begin{equation} \label{3.5}
h_+(\zeta) = h(\zeta) = -h_-(\zeta) \quad \mbox{for a.e. }\zeta\in
L.
\end{equation}
Consider a contour $\Lambda$ which consists of $N$ simple closed
curves, one around each of the arcs $\gamma(a_j, b_j)$. Then
$$\frac{1}{2\pi i} \int_{\Lambda} \frac{h(t)}{t-z}\,dt  = h(z)$$
for $z$ in the exterior of $\Lambda$, and the integral equals zero
for $z\in L$. If we take $z\in \C\setminus L$ and shrink $\Lambda$
to $L$, then
$$h(z) = \frac{1}{\pi i}\int_L \frac{h(t)}{t-z}\,dt,
\quad z\in \C\setminus L,$$ by \eqref{3.5}. Using Privalov's
Fundamental Lemma, we obtain that
$$ S_L h(z) = \frac{1}{\pi i}\int_L \frac{h(t)}{t-z}\,dt =
\frac{h_+(z)}{2} + \frac{h_-(z)}{2} \quad \mbox{for a.e. }z\in L.$$
But the right hand side is zero for a.e. $z\in L$ by \eqref{3.5},
and \eqref{3.4} is proved. Thus we showed that every function
$h=P_{N-1}/\sqrt{R}$ is a solution of the equation $S_L f = 0.$

Suppose now that $f$ satisfies $S_L f(z) = 0$ for a.e. $z\in L.$
Then \eqref{2.3} gives that
\begin{equation} \label{3.6}
C_L f_+ + C_L f_- = 0 \quad \mbox{and} \quad C_L f_+ - C_L f_- = f
\quad \mbox{a.e. on }L,
\end{equation}
where
\[
C_L f(z) := \frac{1}{2\pi i}\int_L \frac{f(t)}{t-z}\,dt, \quad z \in
\C\setminus L.
\]
It follows from the first equation of \eqref{3.6} that
$(\sqrt{R}\,C_L f)_+ - (\sqrt{R}\,C_L f)_- = 0$ and $(\sqrt{R}\,C_L
f)_+ = (\sqrt{R}\,C_L f)_-$ a.e. on $L.$ Clearly, these boundary
values belong to $L^1(L,ds)$ by \eqref{3.6}. This allows us to show
in the usual way that the integral of $\sqrt{R}\,C_L f$ over any
closed contour (even intersecting L) is zero. Hence the analytic in
$\C\setminus L$ function $\sqrt{R}\,C_L f$ can be continued
analytically to the whole $\C$ by a Morera-type theorem of Zalcman
\cite[Th. 1]{Zal}. Observe that $\sqrt{R(z)}\,C_L f(z) =
O\left(z^{N-1}\right)$ as $z\to\infty,$ which implies that this
function is a polynomial $P_{N-1}\in\C_{N-1}[z].$ Finally, we apply
the second equations of \eqref{3.6} and of \eqref{3.4} to find that
\[
f(z) = C_L f_+(z) - C_L f_-(z) = \frac{1}{2}\frac{P_{N-1}(z)}
{\sqrt{R(z)}} + \frac{1}{2}\frac{P_{N-1}(z)}{\sqrt{R(z)}} =
\frac{P_{N-1}(z)}{\sqrt{R(z)}}
\]
for a.e. $z\in L.$

\end{proof}

\begin{proof}[Proof of Theorem \ref{thm2.2}]

Consider the function
\[
f_0(z) = \frac{1}{\pi i \sqrt{R(z)}} \int_{L}
\frac{g(t)\sqrt{R(t)}\,dt}{t-z} = \frac{S_L
(g\sqrt{R})(z)}{\sqrt{R}(z)},\quad z\in L,
\]
which is obtained from \eqref{2.5} by setting $P_{N-1}\equiv 0.$ We
proceed by first proving that $f_0$ is a solution of $S_L f = g.$
Since $g\sqrt{R}\in L^p(L,ds),\ 2<p<\infty,$ we also have that $S_L
(g\sqrt{R})\in L^p(L,ds),$ see \cite{Dav}. It is not difficult to
see that $1/\sqrt{R}\in L^q(L,ds)$ for all $q<2$ (cf. \cite{Sal2},
for example). Using H\"older's inequality, we immediately conclude
that $f_0\in L^r(L,ds)$ for some $r>1.$ Hence $S_L f_0\in
L^r(L,ds)$. Consider the analytic function
\[
F(z):=\frac{1}{2\pi i} \int_{L} \frac{g(t)\sqrt{R(t)}\,dt}{t-z} =
C_L (g\sqrt{R})(z),\quad z\in \C\setminus L.
\]
The Plemelj-Sokhotski formulas \eqref{2.3} read in this case
\[
F_+ - F_- = g\sqrt{R} \quad \mbox{and} \quad F_+ + F_- = f_0\sqrt{R}
\quad \mbox{a.e. on }L.
\]
If we define $\Phi(z):=F(z)/\sqrt{R(z)}, \ z\in \C\setminus L,$ and
use $(\sqrt{R})_+(z)=\sqrt{R(z)}=-(\sqrt{R})_-(z),\ z\in L,$ then we
obtain
\begin{align} \label{3.7}
\Phi_+ + \Phi_- = g \quad \mbox{and} \quad \Phi_+ - \Phi_- = f_0
\quad \mbox{a.e. on }L.
\end{align}
But we also have $(C_Lf_0)_+ - (C_Lf_0)_- = f_0$ a.e. on $L$ for the
Cauchy transform of $f_0,$ see \eqref{2.3}. It follows that the
function $H:=\Phi-C_Lf_0$ is analytic in $\overline\C\setminus L,$
with $H(\infty)=0,$ and satisfies $H_+ - H_- = 0$ a.e. on $L.$ We
can now argue in the same way as in Lemma \ref{lem3.1}, and use a
Morera-type theorem \cite{Zal} to deduce that $H$ can be continued
to an entire function. Thus this function is identically $0$ in
$\overline\C$ by Liouville's theorem. Since $\Phi = C_Lf_0$ in
$\C\setminus L,$ we have by \eqref{2.3} and \eqref{3.7} that
\[
S_Lf_0 = (C_Lf_0)_+ + (C_Lf_0)_- = \Phi_+ + \Phi_- = g \quad
\mbox{a.e. on }L.
\]

If $f$ is any solution of $S_Lf=g,$ then $h=f-f_0$ is a solution of
the homogeneous equation $S_L h = 0.$ Hence it has the form
$h=P_{N-1}/\sqrt{R}$ by Lemma \ref{lem3.1}.

\end{proof}

\begin{proof}[Proof of Corollary \ref{cor2.3}]

We shall compute the Cauchy transform of $f_0$ to show that $f_0$
always satisfies a certain modified integral equation, which is
found below. Consider
\[
C_Lf_0(z) = \frac{1}{2\pi i} \int_L \frac{f_0(t)\,dt}{t-z}, \quad
z\in\C\setminus L,
\]
and define
\[
\Psi(z) := \frac{\sqrt{R(z)}}{2\pi i} \int_{L}
\frac{g(t)\,dt}{\sqrt{R(t)}(t-z)} = \sqrt{R(z)}\,
C_L(g/\sqrt{R})(z), \quad z\in\C\setminus L.
\]
Since $(\sqrt{R})_+ = \sqrt{R} = - (\sqrt{R})_-$ on $L$ and $S_L
(g/\sqrt{R}) = C_L(g/\sqrt{R})_+ + C_L(g/\sqrt{R})_-$ a.e. on $L$ by
\eqref{2.3}, we have that
\begin{align} \label{3.8}
f_0 &= \sqrt{R}\, S_L (g/\sqrt{R}) = \sqrt{R}\,C_L(g/\sqrt{R})_+ +
\sqrt{R}\,C_L(g/\sqrt{R})_- \\ \nonumber &= \Psi_+ - \Psi_-
\end{align}
holds a.e. on $L.$ Passing to the contour integral over both sides
of the cut $L$ in the plane, we obtain
\[
C_Lf_0(z) = \frac{1}{2\pi i} \int_L \frac{\left(\Psi_+(t) -
\Psi_-(t)\right)\,dt}{t-z} = \frac{1}{2\pi i} \oint_L
\frac{\Psi(t)\,dt}{t-z}, \quad z\in\C\setminus L.
\]
Let $\Lambda$ be a contour consisting of $N$ simple closed curves,
one around each of the arcs of $L$, such that $z$ is outside
$\Lambda.$ Cauchy's integral theorem and the definition of $\Psi$
give that
\begin{align*}
C_Lf_0(z) &= \frac{1}{2\pi i} \int_{\Lambda} \frac{\Psi(t)\,dt}{t-z}
= \frac{1}{2\pi i} \int_{\Lambda} \frac{\sqrt{R(t)}}{t-z} \left(
\frac{1}{2\pi i} \int_{L} \frac{g(w)\,dw}{\sqrt{R(w)}(w-t)}\right)
dt \\ &= \frac{1}{2\pi i} \int_{L} \frac{g(w)}{\sqrt{R(w)}} \left(
\frac{1}{2\pi i} \int_{\Lambda} \frac{\sqrt{R(t)}\,dt}{(t-z)(w-t)}
\right) dw.
\end{align*}
Next, we use residues at $z$ and at $\infty$ to evaluate the inner
integral. For the residue at $z$, we immediately obtain
$\sqrt{R(z)}/(w-z).$ Writing $\sqrt{R(t)} = Q(t) +O(1/t)$ near
infinity, where $Q(t)$ is a polynomial of degree $N$, it follows
that
\[
\frac{\sqrt{R(t)}}{t-z} = Q(z,t) + O\left(\frac{1}{t}\right)
\quad\mbox{as }t\to\infty,
\]
with
\begin{align} \label{3.9}
Q(z,t) := \frac{Q(z)-Q(t)}{z-t}.
\end{align}
Note that $Q(z,t)$ is a polynomial in both variables $z$ and $t$, of
degree $N-1.$ Hence the residue at infinity for the inner integral
is equal to $Q(z,w)$, and we obtain that
\[
\frac{1}{2\pi i} \int_{\Lambda} \frac{\sqrt{R(t)}\,dt}{(t-z)(w-t)} =
\frac{\sqrt{R(z)}}{w-z} + Q(z,w).
\]
Finally,
\[
C_Lf_0(z) = \frac{\sqrt{R(z)}}{2\pi i} \int_{L}
\frac{g(w)\,dw}{\sqrt{R(w)}\,(w-z)} + \frac{1}{2} P(z), \quad z\in
\C\setminus L,
\]
where
\begin{align} \label{3.10}
P(z) := \frac{1}{\pi i} \int_{L} \frac{g(w)Q(z,w)\,dw} {\sqrt{R(w)}}
\end{align}
is a polynomial of degree at most $N-1.$ Recall that $(\sqrt{R})_+ =
\sqrt{R} = - (\sqrt{R})_-$ on $L$ and $g/\sqrt{R} =
C_L(g/\sqrt{R})_+ - C_L(g/\sqrt{R})_-$ a.e. on $L$ by \eqref{2.3}.
Thus
\[
\left(C_Lf_0\right)_+ + \left(C_Lf_0\right)_- = \sqrt{R} \left(
C_L(g/\sqrt{R})_+ - C_L(g/\sqrt{R})_- \right) + P = g + P
\]
a.e. on $L.$ On the other hand, \eqref{2.3} directly gives that
\[
\left(C_Lf_0\right)_+ + \left(C_Lf_0\right)_- = S_Lf_0
\]
a.e. on $L.$ We established in this manner that $f_0$ satisfies the
modified integral equation
\begin{align} \label{3.11}
S_Lf_0 = g + P,
\end{align}
where $P$ is defined by \eqref{3.10}. Obviously, $f_0$ is a solution
of the original equation $S_Lf=g$ if and only if $P\equiv 0.$ Thus
it remains to show that the vanishing of $P$ is equivalent to
\eqref{2.7}. Observe from \eqref{3.9} that
\[
Q(z,t) = q_0(t) z^{N-1} + q_1(t) z^{N-2} + \ldots + q_{N-1}(t),
\]
where each $q_k$ is a monic polynomial of degree $k.$ Then $P\equiv
0$ is equivalent to the system
\[
\frac{1}{2\pi i} \int_{L} \frac{g(w)q_k(w)\,dw} {\sqrt{R(w)}} = 0,
\quad k=0,\ldots,N-1,
\]
by \eqref{3.10}. Noting that the polynomials $q_k$ are linearly
independent, we conclude that the above system in equivalent to
\eqref{2.7}.

\end{proof}

\begin{proof}[Proof of Corollary \ref{cor2.4}]

It is clear that \eqref{2.5} may be written in the form
\[
f = \frac{S_{L}(g\sqrt{R})}{\sqrt{R}} + \frac{P_{N-1}}{\sqrt{R}},
\]
where $P_{N-1}\in\C_{N-1}[z].$ Let $E\subset L\setminus
\{a_j,b_j\}_{j=1}^N$ be compact. In order to prove that $f\in
H_{\alpha}\left(E\right),$ it is sufficient to show that $S_{L}
(g\sqrt{R})\in H_{\alpha}\left(E\right).$ The function
$G:=g\sqrt{R}\in H_{\alpha}\left(L\right)$ can be continued to a
function $\tilde G\in H_{\alpha}\left(\Gamma\right),$ where $\Gamma$
is a closed Ahlfors regular curve, so that $S\tilde G \in
H_{\alpha}\left(\Gamma\right)$ by \cite{Sal,Gus}. But the Cauchy
singular integral $S$ of $\tilde G \vert_{\Gamma\setminus L}$ is
analytic at every $z\in E$. Thus it follows that $S_{L} G \in
H_{\alpha}\left(E\right).$

\end{proof}

\begin{lemma} \label{lem3.2}
Let $L=\cup_{j=1}^N \gamma(a_j,b_j)$ be a union of disjoint Ahlfors
regular arcs. If $g\in H_{\alpha}(L),\ \alpha>0,$ then the function
$f_0$ defined in \eqref{2.6} is continuous on $L$, and
$f_0(a_j)=f_0(b_j)=0,\ j=1,\ldots,N.$
\end{lemma}

\begin{proof}

We observe that $f_0 = \sqrt{R}\,S_{L}(g/\sqrt{R}),$ and that
$g/\sqrt{R}$ is H\"older continuous on any compact set $E\subset L$
that does not include the endpoints of $L.$ Using a similar argument
with continuation to a H\"older continuous function on the closed
curve $\Gamma$, such as in the above proof of Corollary
\ref{cor2.4}, we conclude that $S_{L}(g/\sqrt{R})$ is also H\"older
continuous on $E$. Hence we now need to analyze the behavior of
$S_{L}(g/\sqrt{R})$ near the endpoints of $L.$ This analysis was
already carried out in Chapter 4 of \cite{Mus} for smooth (or
piecewise smooth) $L$. Since the argument is rather technical, and
requires relatively small adjustments for the case of Ahlfors
regular $L$, we do not reproduce it here. In particular, it is shown
in \cite{Mus} (see equations (29.8) and (29.9) on page 75) and in
\cite{Sal2} that
\[
S_{L}(g/\sqrt{R})(z) = G(z)/|z-c|^{\beta},\quad \beta<1/2,
\]
for $z\in L$ near any of the endpoints $c$ of $L$, where $G$ is
H\"older continuous on $L.$ It is immediate that
\[
f_0(z) = G(z)\,|z-c|^{1/2-\beta}
\]
for $z\in L$ near $c$, so that $f_0$ is H\"older continuous on $L$
and $f(a_j)=f(b_j)=0,\ j=1,\ldots,N.$

\end{proof}

\begin{proof}[Proof of Corollary \ref{cor2.5}]

Suppose that there exists a bounded function $f$ that satisfies
$S_Lf=g$ on $L.$ We know from the proof of Corollary \ref{cor2.3}
that $f_0$ defined in \eqref{2.6} satisfies the equation $S_Lf_0 = g
+ P$ on $L$, where $P$ is a polynomial of degree at most $N-1$, see
\eqref{3.10}-\eqref{3.11}. Defining $h:=f_0-f,$ we readily have that
$S_L h = P.$ It will be shown below that $h\equiv 0$, so that
$P\equiv 0$ and $S_L f_0 = g$. But then  \eqref{2.7} holds by
Corollary \ref{cor2.3}.

Consider the equation $S_L h = P.$ All solutions of this equation
are described by \eqref{2.5}:
\begin{align} \label{3.12}
h(z) = \frac{1}{\pi i \sqrt{R(z)}} \int_{L}
\frac{P(t)\sqrt{R(t)}\,dt}{t-z} + \frac{P_{N-1}(z)}{\sqrt{R(z)}}
\quad \mbox{a.e. on }L,
\end{align}
where $P_{N-1}\in\C_{N-1}[z]$ is arbitrary. We evaluate the integral
$S_L (P\sqrt{R})$ in the above formula by following an idea used in
the proof of Corollary \ref{cor2.3}. Recall that $(\sqrt{R})_+ =
\sqrt{R} = - (\sqrt{R})_-$ on $L$ and $S_L (P\sqrt{R}) =
C_L(P\sqrt{R})_+ + C_L(P\sqrt{R})_-$ a.e. on $L$ by \eqref{2.3}. We
find the Cauchy transform $C_L(P\sqrt{R})$ by passing to the contour
integral over both sides of the cut $L$ in the plane. This yields
\[
C_L(P\sqrt{R})(z) = \frac{1}{2\pi i} \int_L
\frac{P(t)\sqrt{R(t)}\,dt}{t-z} = \frac{1}{4\pi i} \oint_L
\frac{P(t)\sqrt{R(t)}\,dt}{t-z}, \quad z\in\C\setminus L,
\]
where in the second integral we have the {\em boundary limit values}
of $P(t)\sqrt{R(t)}$ on $L$ (from $\C\setminus L).$ Let $\Lambda$ be
again a contour consisting of $N$ simple closed curves, one around
each of the arcs of $L$, such that $z$ is outside $\Lambda.$ Using
Cauchy's integral theorem, we obtain that
\begin{align*}
C_L(P\sqrt{R})(z) = \frac{1}{4\pi i} \int_{\Lambda}
\frac{P(t)\sqrt{R(t)}\,dt}{t-z}, \quad z\in\C\setminus L.
\end{align*}
The latter integral is found by evaluating the residues of the
integrand at $z$ and at $\infty$. The residue at $z$ is clearly
equal to $P(z)\sqrt{R(z)}/2.$ Writing $P(t)\sqrt{R(t)} = T(t)
+O(1/t)$ near infinity, where $T(t)$ is a polynomial of degree at
most $2N-1$, we find that the residue at $\infty$ is equal to
$T(z)/2.$ Hence
\[
C_L(P\sqrt{R})(z) = \frac{P(z)\sqrt{R(z)}}{2} + \frac{T(z)}{2},
\quad z\in\C\setminus L,
\]
and \eqref{2.3} gives
\[
S_L(P\sqrt{R})(z) = T(z), \quad z\in L,
\]
because $(P\sqrt{R})_+ = - (P\sqrt{R})_-$ on $L$. Returning to
\eqref{3.12}, we have
\[
h(z) = \frac{T(z)+P_{N-1}(z)}{\sqrt{R(z)}} \quad \mbox{a.e. on }L.
\]
Note that the numerator is a polynomial of degree $2N-1,$ which has
to vanish at $2N$ endpoints of $L$ in order for $h$ to be bounded on
$L$. Therefore, this polynomial is identically zero, with immediate
implications that $h\equiv 0$, $f_0\equiv f$, $P\equiv 0$ and $S_L
f_0 = g$. Thus \eqref{2.7} holds by Corollary \ref{cor2.3}, and we
showed that $f_0$ is the unique bounded solution of $S_Lf=g.$

Conversely, assume that \eqref{2.7} is satisfied. Corollary
\ref{cor2.3} shows that $f_0$ of \eqref{2.8} (or of \eqref{2.6}) is
a solution of $S_Lf=g.$ Applying Lemma \ref{lem3.2}, we obtain that
$f_0$ is continuous on $L$, and it vanishes at the endpoints of $L.$
Hence it is bounded on $L$, and the uniqueness of such solution
follows from the first part of this proof.

\end{proof}

\bibliographystyle{amsplain}

\end{document}